\documentclass[12pt]{amsart}

\usepackage{color}
\usepackage{amsmath,amsthm}
\usepackage{lipsum}

\usepackage{tikz}
\usetikzlibrary{arrows}
\usetikzlibrary{patterns}

\newtheorem{theorem}{Theorem}[section]

\newtheorem{lemma}[theorem]{Lemma}
\newtheorem{corollary}[theorem]{Corollary}

\theoremstyle{definition}
\newtheorem{definition}[theorem]{Definition}

\newtheorem{example}[theorem]{Example}

\newtheorem{conjecture}[theorem]{Conjecture}

\newcommand{\NN}{ \ensuremath{\mathbb{N}}}

\newcommand{\FF}{ \ensuremath{\mathbb{F}}}

\newcommand{\rank}{\mathop{\mathrm{rank}}}

\newcommand{\aaa}{\mathbf{a}}

\newcommand{\st}{\mathrm{st}}

\def\cocoa{{\hbox{\rm C\kern-.13em o\kern-.07em C\kern-.13em o\kern-.15em A}}}

\newcommand{\cont}{\ensuremath{\mathcal{C}_{p \to q}}}
\newcommand{\bcont}{\ensuremath{\mathcal{C}^{(b)}_{p \to q}}}

\newcommand{\marrow}{\marginpar{$\longleftarrow$}}

\newcommand{\martina}[1]{\textsc{\textcolor{blue}{Martina says:}} \marrow \textsf{#1}}

\begin{document}

\title{Lefschetz properties of balanced 3-polytopes}


\author[D. Cook II]{David Cook II}
\address{David Cook II,
Department of Mathematics \& Computer Science,
Eastern Illinois University,
Charleston, IL 61920 USA}
\email{dwcook@eiu.edu}

\author[M. Juhnke-Kubitzke]{Martina Juhnke-Kubitzke}
\address{Martina Juhnke-Kubitzke,
Institut f\"ur Mathematik, Universit\"at Osnabr\"uck, 49069 Osnabr\"uck, Germany}
\email{juhnke-kubitzke@uos.de}
\thanks{Research of Juhnke-Kubitzke was partially supported by the German Research Council DFG-GRK 1916.}

\author[S. Murai]{Satoshi Murai}
\address{
Satoshi Murai,
Department of Pure and Applied Mathematics,
Graduate School of Information Science and Technology,
Osaka University,
Suita, Osaka, 565-0871, Japan}
\email{s-murai@ist.osaka-u.ac.jp
}
\thanks{Research of Murai was partially supported by KAKENHI16K05102.
}

\author[E. Nevo]{Eran Nevo}
\address{Eran Nevo, Einstein Institute of Mathematics, The Hebrew University of Jerusalem, Jerusalem, Israel}
\email{nevo@math.huji.ac.il}
\thanks{Research of Nevo was partially supported by Israel Science Foundation grant ISF-1695/15.
}



\begin{abstract}
In this paper, we study Lefschetz properties of Artinian reductions of Stanley--Reisner rings of balanced simplicial $3$-polytopes.
A $(d-1)$-dimensional simplicial complex is said to be {\em balanced} if its graph is $d$-colorable.
If a simplicial complex is balanced, then its Stanley--Reisner ring has a special system of parameters induced by the coloring.
We prove that the Artinian reduction of the Stanley--Reisner ring of a balanced simplicial $3$-polytope with respect to this special system of parameters has the strong Lefschetz property if the characteristic of the base field is not two or three.
Moreover, we characterize $(2,1)$-balanced simplicial polytopes, i.e., polytopes with exactly one red vertex and two blue vertices in each facet, such that an analogous property holds. In fact, we show that this is the case if and only if the induced graph on the blue vertices satisfies a Laman-type combinatorial condition.
\end{abstract}

\maketitle

\section{Introduction}
Let $\FF$ be an infinite field.
An Artinian Gorenstein standard graded $\FF$-algebra $A=A_0\oplus A_1\oplus \cdots \oplus A_s$ with $A_0 \cong A_s \cong \FF$ is said to have the {\em strong Lefschetz property} (SLP, for short) if there is a linear form $w \in A_1$ such that the multiplication map $\times w^{s-2i} : A_i \to A_{s-i}$ is bijective for all $i < \frac s 2$.
This property is motivated by the Hard Lefschetz theorem, and has been of great interest in both algebra and combinatorics, with a multitude of applications
(see the book \cite{book:Lefschetz}).
Proving the SLP is difficult in general, and
it is interesting to find new classes of Artinian Gorenstein algebras having the SLP.
In this paper, we study the SLP for certain Artinian reductions of the Stanley--Reisner rings of simplicial $3$-polytopes, which satisfy nice vertex coloring conditions.

Given a simplicial complex $\Delta$ on the vertex set $V$,
the ideal $I_\Delta$ of $\FF[x_v:v\in V]$, defined by
$$I_\Delta=(x_{v_1} \cdots x_{v_k}: \{ v_1,\dots,v_k\} \subseteq V,\ \{ v_1,\dots,v_k\} \not \in \Delta),$$
is called the {\em Stanley--Reisner ideal} of $\Delta$, and
the quotient ring
$$\FF [\Delta] = \FF[x_v : v \in V] /I_\Delta$$
is called the {\em Stanley--Reisner ring} of $\Delta$ over the field $\FF$.
A $(d-1)$-dimensional simplicial complex $\Delta$ is said to be {\em balanced} (or {\em completely balanced} in some literature) if its graph is $d$-colorable,
equivalently, if there is a map $\kappa : V \to [d]=\{1,2,\dots,d\}$ such that, for all faces $\sigma \in \Delta$, one has $|\{ v \in \sigma: \kappa(v)=i\}| \leq 1$ for all $i \in [d]$.
It was proved by Stanley \cite{Stan:79} that if $\Delta$ is balanced, then the sequence of linear forms $\Theta=(\theta_1,\dots,\theta_d$) defined by $\theta_i= \sum_{\kappa (v)=i} x_v$ for $i=1,2,\dots,d$,
is a system of parameters for $\FF[\Delta]$.
We call such
$\Theta$
a {\em colored system of parameters} (colored s.o.p.\ for short) for $\FF[\Delta]$.
Note that, if $\Delta$ is strongly connected, then 
a map $\kappa$ satisfying the above condition is unique up to permutation of the elements of $[d]$ (see Section 2). So, as a set, the colored s.o.p.\ does not depend on the choice of the coloring $\kappa$.

A {\em simplicial $d$-sphere} is a simplicial complex which is homeomorphic to a $d$-sphere.
In general, the boundary complex of a simplicial $d$-polytope is a simplicial $(d-1)$-sphere,
and, by a classical theorem of Steinitz, every simplicial $2$-sphere is the boundary complex of some simplicial $3$-polytope.
If $\Delta$ is the boundary complex of a simplicial polytope and $\Theta$ is a linear system of parameters for $\FF[\Delta]$,
then the algebra $\FF[\Delta]/\Theta \FF[\Delta]$ is an Artinian Gorenstein algebra,
and moreover, by the Hard Lefschetz theorem for projective toric varieties,
for a certain choice of $\Theta$ (corresponding to convex embeddings, or generic ones) this algebra has the SLP in characteristic $0
$ (see \cite[III, Section 1]{Stan:Gr}).
However, when $\Delta$ is balanced,
the linear system of parameters $\Theta$ used in this setting is not the colored s.o.p.\ defined above, and it is hence natural to ask whether the SLP holds for this specific s.o.p.\ as well.

We say that a balanced simplicial sphere $\Delta$ has the {\em colored SLP} over a field $\FF$ if $\FF[\Delta]/\Theta \FF[\Delta]$ has the SLP for the colored s.o.p.\ $\Theta$ for $\FF[\Delta]$.
The first main result of this paper is the following.

\begin{theorem}
\label{thm:main1}
Let $\FF$ be an infinite field with $\mathrm{char}(\FF) \ne 2,3$.
Any balanced simplicial $2$-sphere has the colored SLP over $\FF$.
\end{theorem}

Note that in characteristic $2$ and $3$ \emph{any} $\omega\in (\FF[\Delta]/\Theta \FF[\Delta])_1$ satisfies $\omega^3=0$ for $\Theta$ the colored s.o.p., and thus $\Delta$ fails to have the colored SLP over $\FF$.
We consider a similar problem also for a more general class of spheres, namely, $(2,1)$-balanced simplicial $2$-spheres.
For $\aaa=(a_1,\dots,a_n) \in \NN^n$,
a simplicial complex $\Delta$ on the vertex set $V$ is said to be {\em $\aaa$-balanced}
if $\Delta$ has dimension $a_1+ \cdots +a_n-1$
and there is a map $\kappa: V \to [n]$ such that, for any face $\sigma \in \Delta$,
we have $|\{ v \in \sigma: \kappa(v)=i\}| \leq a_i$ for all $i \in [n]$.
We call such a map $\kappa$ an {\em $\aaa$-coloring} of $\Delta$.
By a result of Stanley \cite{Stan:79}, for an $\aaa$-balanced simplicial complex $\Delta$, there exists an s.o.p.\ $\theta_1,\dots,\theta_d$ such that exactly $a_j$ of the
$\theta_i$'s are a linear combination of the variables $x_v$ having the same color $j$ (that is, $\kappa(v)=j$).
We call such a system of parameters an {\em $\aaa$-colored system of parameters} ({\em $\aaa$-colored s.o.p.\ }for short) for $\FF[\Delta]$.

It is natural to ask if an analogue of Theorem \ref{thm:main1} holds for $\aaa$-balanced simplicial polytopes and spheres. Somewhat surprisingly, we find that the answer is negative even when $\aaa=(2,1)$. More precisely, we provide the following combinatorial characterization of the SLP for Artinian reductions of $\FF[\Delta]$ with respect to any $(2,1)$-colored s.o.p.\ if $\Delta$ is a $(2,1)$-balanced simplicial sphere.

\begin{theorem}
\label{thm:main2}
Let $\FF$ be an infinite field with $\mathrm{char}(\FF) \ne 2,3$.
Let $\Delta$ be a $(2,1)$-balanced simplicial $2$-sphere, $\kappa:V\rightarrow \{1,2\}$ a $(2,1)$-coloring of $\Delta$, and $U$ the set of the vertices $v$ of $\Delta$ with $\kappa (v)=1$.
The following conditions are equivalent.
\begin{itemize}
\item[(i)] There is a $(2,1)$-colored s.o.p.\ $\Theta$ for $\FF[\Delta]$ such that $\FF[\Delta]/(\Theta)$ has the SLP.
\item[(ii)]
For any subset $W \subseteq U$ with $|W| \geq 2$, the induced subcomplex $\Delta_W=\{ \sigma \in \Delta: \sigma \subseteq W\}$ has at most $2|W|-3$ edges.
\end{itemize}
\end{theorem}

The above criterion (ii) is motivated by, and essentially the same as, Laman's criterion for minimal generic rigidity of graphs in the plane \cite{Laman}.

Theorem \ref{thm:main2} allows us to construct $(2,1)$-balanced simplicial $2$-spheres $\Delta$ such that the Artinian reduction of $\FF[\Delta]$ with respect to any $(2,1)$-colored s.o.p.\ fails to have the SLP
(see Example \ref{4.2}).

Even though an analogue of Theorem \ref{thm:main1} for $\aaa$-balanced simplicial polytopes
does not hold, considering Theorem \ref{thm:main1}
we propose the following conjecture in higher dimensions:

\begin{conjecture}
Any balanced simplicial sphere (or at least any balanced simplicial polytope) has the colored SLP over a field of characteristic $0$.
\end{conjecture}

The paper is structured as follows. Section \ref{sect:prel} provides some background on simplicial complexes and constructions on simplicial spheres.
Section \ref{sect:Lefschetz} contains the proof of our first main result Theorem \ref{thm:main1}. Finally, Section \ref{sect:aBalanced} studies $(2,1)$-balanced simplicial $2$-spheres. Our second main result (Theorem \ref{thm:main2}) characterizes when those have the SLP with respect to a $(2,1)$-colored s.o.p.


\section{Preliminaries}\label{sect:prel}
In this section we provide some background and introduce notation that will be used throughout this article.

\subsection{Simplicial complexes}
A simplicial complex $\Delta$ on a finite set $V$ is a collection of subsets of $V$ that is closed under inclusion.
An element of $\Delta$ is called a {\em face} of $\Delta$ and maximal faces (under inclusion) are called {\em facets} of $\Delta$.
The {\em dimension} of a face is its cardinality minus one, and the dimension of a simplicial complex is the maximal dimension of its faces.
Faces of dimension $0$ are called {\em vertices}
and faces of dimension $1$ are called {\em edges}.
We denote by $V(\Delta)=\{v: \{v\} \in \Delta\}$ the vertex set of $\Delta$, and identify a singleton $\{v\} \in \Delta$ with $v \in V(\Delta)$.
A simplicial complex is said to be {\em pure} if all its facets have the same dimension.
A pure simplicial complex $\Delta$ is said to be {\em strongly connected} if, for any pair $\sigma,\tau$ of facets of $\Delta$, there is a sequence $\rho_1,\dots,\rho_k$ of facets of $\Delta$ such that $|\sigma \setminus \rho_1| = |\rho_1 \setminus \rho_2|= \cdots = |\rho_k\setminus \tau|=1$.

For a simplicial complex $\Delta$, a map $\kappa :V(\Delta) \to [d]$ is said to be a {\em proper $d$-coloring} of $\Delta$ if $\kappa(u)\ne \kappa(v)$ for all edges $\{u,v\} \in \Delta$.
Note that, a $(d-1)$-dimensional simplicial complex $\Delta$ is balanced if and only if it has a proper $d$-coloring.
If, in addition, $\Delta$ is strongly connected, then the choice of a proper $d$-coloring is unique up to permutations of the elements of $[d]$
(since the values of $\kappa$ for vertices of one facet determine the values of $\kappa$ for all other vertices).
The smallest example of a balanced simplicial $(d-1)$-sphere is the boundary complex of
the \emph{$d$-crosspolytope},
which is the convex hull of the unit vectors and their antipodes in $\mathbb{R}^d$.

For a simplicial complex $\Delta$ and a vertex $v \in V(\Delta)$, the simplicial complex
$$\st_\Delta(v)=\{ \tau \in \Delta: \tau \cup \{v\} \in \Delta\}$$
is called the {\em star} of $v$ in $\Delta$.
A {\em simplicial $2$-ball} is a simplicial complex that is homeomorphic to a $2$-dimensional ball.
If $\Delta$ is a simplicial $2$-sphere, then $\st_\Delta(v)$ is a simplicial $2$-ball for any vertex $v \in V(\Delta)$.
For a simplicial $2$-ball $B$, we write $\partial B$ for the boundary complex of $B$, and write $\mathrm{int} (B)=B \setminus \partial B$ for the set of all interior faces of $B$.

Given a $(d-1)$-dimensional simplicial complex $\Delta$, a sequence of linear forms $\theta_1,\dots,\theta_d \in \FF[\Delta]$ is said to be a {\em linear system of parameters} ({\em l.s.o.p.\ }for short) for $\FF[\Delta]$ if $\dim_\FF (\FF[\Delta]/(\theta_1,\dots,\theta_d) \FF[\Delta]) < \infty$,
and the Artinian algebra $\FF[\Delta]/(\theta_1,\dots,\theta_d) \FF[\Delta]$ is called the {\em Artinian reduction} of $\FF[\Delta]$ w.r.t.\ $\theta_1,\dots,\theta_d$. As mentioned in the introduction,
if $\Delta$ is balanced and $\kappa$ a proper $d$-coloring of $\Delta$,
then the sequence of linear forms $\theta_1,\dots,\theta_d$ defined by $\theta_i=\sum_{v \in V(\Delta),\ \kappa(v)=i} x_v$ forms an l.s.o.p.\ for $\FF[\Delta]$, the so-called {\em colored s.o.p.}

\subsection{Operations on simplicial spheres}
Finally, we recall two combinatorial operations on simplicial $2$-spheres.
For finite subsets $\sigma_1,\dots,\sigma_k$, we write
$$\langle \sigma_1,\dots,\sigma_k \rangle
=\{ \tau: \tau \subseteq \sigma_i \mbox{ for some }i\}$$
for the simplicial complex generated by $\sigma_1,\dots,\sigma_k$.
\begin{definition}
\label{def:sum}
Let $\Delta$ and $\Gamma$ be $2$-dimensional simplicial complexes. If $\Delta \cap \Sigma$ is generated by a single $2$-dimensional face $\sigma$,
then the simplicial complex
$$(\Delta\setminus \{\sigma\}) \cup (\Gamma \setminus \{\sigma\})$$
is called the {\em connected sum} of $\Delta$ and $\Gamma$, and denoted by $\Delta \#_\sigma \Gamma$.
\end{definition}

A {\em missing triangle} of a simplicial complex $\Delta$ is a set $\{a,b,c\}$
such that $\{a,b\}$, $\{a,c\}$, $\{b,c\} \in \Delta$ and $\{a,b,c\} \not \in \Delta$.
The following property is well-known.

\begin{lemma}
\label{missingface}
Let $\Delta$ be a simplicial $2$-sphere.
If $\sigma$ is a missing triangle of $\Delta$, then there are unique simplicial $2$-spheres $\Gamma$ and $\Sigma$ such that $\Delta=\Gamma \#_\sigma \Sigma$.
Moreover, if $\Delta$ is balanced, then so are $\Gamma$ and $\Sigma$.
\end{lemma}

The first part of Lemma \ref{missingface} follows easily from Jordan's curve theorem; see \cite[Lemma 1.3]{BD} for a more general statement for PL-manifolds.
The second part follows from the fact that the $1$-skeleta of $\Gamma$ and $\Sigma$ are subgraphs of the one of $\Delta$.

\begin{definition}
\label{def:contraction}
For a simplicial complex $\Delta$ and two of its vertices $p,q$,
we define
$$\cont(\Delta)= \{ \sigma \in \Delta: p \not \in \sigma\} \cup \{ (\sigma \setminus \{p\}) \cup \{q\}: p \in \sigma \in \Delta\}.$$
If $\{p,q\}$ is an edge of $\Delta$,
then the operation $\Delta \to \cont(\Delta)$ is called the {\em contraction} of the edge $\{p,q\}$.
\end{definition}

For a simplicial $2$-sphere $\Delta$, that is not the boundary of a $3$-simplex,
a contraction $\Delta \to \cont(\Delta)$ is {\em admissible} if there are no missing triangles of $\Delta$ that contain the edge $\{p,q\}$.
Note that this condition is equivalent to saying that $\st_\Delta(p) \cap \st_\Delta(q)=\langle \{p,q,s\},\{p,q,t\}\rangle$ for some distinct vertices $s,t$.
The following fact is well-known, see e.g., \cite[Lemma 1]{Bar-RP2} for a short proof, or, more generally, \cite[Theorem 1.4]{Nevo-VK} for edge contractions in PL-manifolds.

\begin{lemma}
\label{admissiblecontraction}
Let $\Delta$ be a simplicial $2$-sphere. If $\Delta \to \cont (\Delta)$ is an admissible contraction, then $\cont(\Delta)$ is a simplicial $2$-sphere.
\end{lemma}

\subsection{Contractions for balanced simplicial $2$-spheres}



It is a classical result in graph theory, sometimes called the Three Color Theorem, that a simplicial $2$-sphere is balanced if and only if each of its vertices has an even degree. See \cite[p.44--46]{Golovina-Yaglom} for maybe the earliest published complete proof.
For such simplicial spheres, the following contraction operation has been considered.

\begin{definition}
\label{def:balancedcontraction}
Let $\Delta$ be a balanced simplicial $2$-sphere.
We say that a pair $(p,q)$ of distinct vertices of $\Delta$ is {\em a contractible pair} in $\Delta$ if
\begin{itemize}
\item[(i)] $p$ and $q$ have the same color, that is, $\kappa(p) = \kappa(q)$ for some proper $3$-coloring $\kappa$ of $\Delta$, and
\item[(ii)] there are vertices $s,t,w$ such that
$$\st_\Delta(p) \cap \st_\Delta(q)=\langle \{s,w\},\{w,t\} \rangle.$$
\end{itemize}
For a contractible pair $(p,q)$, we define
$$ \bcont(\Delta) = \big(\Delta \setminus \mathrm{int} \big(\st_\Delta(p) \cup \st_\Delta(q)\big)\big) \cup \big\{ \sigma \cup \{q\}: \sigma \in \partial \big(\st_\Delta(p) \cup \st_\Delta(q)\big)\big\}.$$
The operation $\Delta \to \bcont(\Delta)$ is called the {\em balanced contraction} (or $4$-contraction in some literatures) of the pair $(p,q)$, see Figure 1
for an illustration.
\end{definition}
\begin{center}
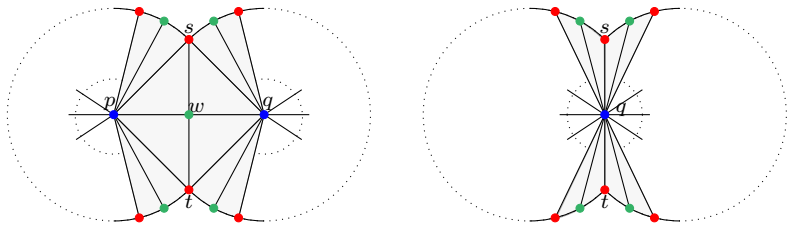
\begin{figure}[h]\label{fig:balancedContr}
\begin{tikzpicture}[scale=1.0, every node/.style={font=\footnotesize}]
\tikzstyle{point}=[ fill,circle, inner sep=0pt,minimum width=3pt,minimum height=3pt]
\filldraw[fill=gray!20!, opacity=0.3] (3,4) to (4,3) to (5,4) to (4,5) to (3,4) to (5,4);
\filldraw[fill=gray!20!, opacity=0.3]
(3,4) -- (4,5) arc (45:76:1.4142) -- cycle;
\filldraw[fill=gray!20!,opacity=0.3]
(3,4) -- (4,3) arc (-45:-76:1.4142) -- cycle;
\filldraw[fill=gray!20!, opacity=0.3]
(5,4) -- (4,5) arc (135:104:1.4142) -- cycle;
\filldraw[fill=gray!20!, opacity=0.3]
(5,4) -- (4,3) arc (225:256:1.4142) -- cycle;

\draw (3,4) to (4,3) to (5,4) to (4,5) to (3,4) to (5,4);
\draw (4,3) to (4,5);
\draw (3,4) to (3.34,2.627);
\draw (3,4) to (3.67,2.755);
\draw (3,4) to (3.34,5.373);
\draw (3,4) to (3.67, 5.245);
\draw (5,4) to (4.66,2.627);
\draw (5,4) to (4.33,2.755);
\draw (5,4) to (4.66,5.373);
\draw (5,4) to (4.33,5.245);


\draw (3,4) to (2.4,4);
\draw (3,4) to (2.5,4.3317);
\draw (3,4) to (2.5,3.6683);

\draw (5,4) to (5.6,4);
\draw (5,4) to (5.5,4.3317);
\draw (5,4) to (5.5,3.6683);
\draw[dotted] (4,5) arc (45:315:1.4142);
\draw (4,5) arc (45:90:1.4142);
\draw (3,2.5858) arc (270:315:1.4142);
\draw[dotted] (4,3) arc (225:495:1.4142);
\draw (4,5) arc (135:90:1.4142);
\draw (4,3) arc (225:270:1.4142);
\draw[dotted] (3.12,4.458) arc (75:285:0.5);
\draw[dotted] (4.88,4.458) arc (105:-105:0.5);
\draw (2.95,4.15) node {{\tiny$p$}};
\draw (5.05,4.15) node {{\tiny$q$}};
\draw (4.1,4.1) node {{\tiny$w$}};
\draw (4,5.15) node {{\tiny$s$}};
\draw (4,2.85) node {{\tiny$t$}};
\tikzstyle{point}=[ fill=blue, draw=blue,circle, inner sep=0pt,minimum width=3pt,minimum height=3pt]
\node [point, label={[label distance=0cm]110:}] at (3,4){};
\node [point, label={[label distance=0cm]90:}] at (5,4){};

\tikzstyle{point}=[ fill=red, draw=red,circle, inner sep=0pt,minimum width=3pt,minimum height=3pt]
\node [point, label={[label distance=0cm]90:}] at (4,3){};
\node [point, label={[label distance=0cm]90:}] at (4,5){};
\node [point, label={[label distance=0cm]90:}] at (3.34,2.627){};
\node [point, label={[label distance=0cm]90:}] at (3.34,5.373){};
\node [point, label={[label distance=0cm]90:}] at (4.66,2.627){};
\node [point, label={[label distance=0cm]90:}] at (4.66,5.373){};
\definecolor{darkgreen}{rgb}{0.2, 0.7, 0.4}
\tikzstyle{point}=[ fill=darkgreen, draw=darkgreen,circle, inner sep=0pt,minimum width=3pt,minimum height=3pt]

\node [point, label={[label distance=0cm]0:}] at (4,4){};

\node [point, label={[label distance=0cm]90:}] at (3.67,2.755){};

\node [point, label={[label distance=0cm]90:}] at (3.67,5.245){};

\node [point, label={[label distance=0cm]90:}] at (4.33,2.755){};

\node [point, label={[label distance=0cm]90:}] at (4.33,5.245){};
\end{tikzpicture}\quad
\begin{tikzpicture}
[scale=1.0, every node/.style={font=\footnotesize}]
\tikzstyle{point}=[ fill,circle, inner sep=0pt,minimum width=3pt,minimum height=3pt]

\filldraw[fill=gray!20!, opacity=0.3] (4,4) -- (3.34,5.373) arc (75:44:1.4142)--cycle;
\filldraw[fill=gray!20!, opacity=0.3] (4,4) -- (4,5) arc (135:104:1.4142)--cycle;
\filldraw[fill=gray!20!, opacity=0.3] (4,4) -- (4,3) arc (225:256:1.4142)--cycle;
\filldraw[fill=gray!20!, opacity=0.3] (4,4) -- (4,3) arc (-44:-75:1.4142)--cycle;

\draw (4,3) to (4,5);
\draw (4,4) to (3.67,2.755);
\draw (4,4) to (3.34,5.373);
\draw (4,4) to (3.67, 5.245);
\draw (4,4) to (4.66,2.627);
\draw (4,4) to (4.33,2.755);
\draw (4,4) to (4.66,5.373);
\draw (4,4) to (4.33,5.245);
\draw (4,4) to (3.34,2.627);


\draw (4,4) to (3.4,4);
\draw (4,4) to (3.5,4.3317);
\draw (4,4) to (3.5,3.6683);

\draw (4,4) to (4.6,4);
\draw (4,4) to (4.5,4.3317);
\draw (4,4) to (4.5,3.6683);
\draw[dotted] (4,5) arc (45:315:1.4142);
\draw (4,5) arc (45:90:1.4142);
\draw (3,2.5858) arc (270:315:1.4142);
\draw[dotted] (4,3) arc (225:495:1.4142);
\draw (4,5) arc (135:90:1.4142);
\draw (4,3) arc (225:270:1.4142);
\draw[dotted] (4.5,4) arc (0:65:0.5);
\draw[dotted] (4.5,4) arc (0:-65:0.5);
\draw[dotted] (3.5,4) arc (180:115:0.5);
\draw[dotted] (3.5,4) arc (180:245:0.5);

\draw (4.22,4.07) node {{\tiny$q$}};
\draw (4,5.15) node {{\tiny$s$}};
\draw (4,2.85) node {{\tiny$t$}};

\tikzstyle{point}=[ fill=blue, draw=blue,circle, inner sep=0pt,minimum width=3pt,minimum height=3pt]
\node [point, label={[label distance=0cm]110:}] at (4,4){};

\tikzstyle{point}=[ fill=red, draw=red,circle, inner sep=0pt,minimum width=3pt,minimum height=3pt]
\node [point, label={[label distance=0cm]90:}] at (4,3){};
\node [point, label={[label distance=0cm]90:}] at (4,5){};
\node [point, label={[label distance=0cm]90:}] at (3.34,2.627){};
\node [point, label={[label distance=0cm]90:}] at (3.34,5.373){};
\node [point, label={[label distance=0cm]90:}] at (4.66,2.627){};
\node [point, label={[label distance=0cm]90:}] at (4.66,5.373){};
\definecolor{darkgreen}{rgb}{0.2, 0.7, 0.4}
\tikzstyle{point}=[ fill=darkgreen, draw=darkgreen,circle, inner sep=0pt,minimum width=3pt,minimum height=3pt]

\node [point, label={[label distance=0cm]90:}] at (3.67,2.755){};
\node [point, label={[label distance=0cm]90:}] at (3.67,5.245){};
\node [point, label={[label distance=0cm]90:}] at (4.33,2.755){};
\node [point, label={[label distance=0cm]90:}] at (4.33,5.245){};
\end{tikzpicture}
\caption{The balanced contraction of a pair $(p,q)$ showing also the change in a coloring.}
\end{figure}
\end{center}

Observe that, by the uniqueness of a coloring, the first condition in Definition \ref{def:balancedcontraction} does either hold for any proper $3$-coloring or none.
Note also that since $\st_\Delta(p)$ and $\st_\Delta(q)$ are simplicial $2$-balls, the second condition implies that $\st_\Delta(p) \cup \st_\Delta(q)$ is a simplicial $2$-ball, so its boundary, used in the definition of $\bcont(\Delta)$, is indeed well defined.

It is easy to see that if $\Delta$ is a balanced simplicial $2$-sphere and $(p,q)$ is a contractible pair in $\Delta$, then $\bcont(\Delta)$ is a balanced simplicial $2$-sphere.
The following result was proved by Batagelj \cite{Ba}\footnote{Batagelj phrased his result for simplicial spheres where all vertex degrees are even.}. 




\begin{theorem}
\label{2.7}
Let $\Delta$ be a balanced simplicial $2$-sphere which is not the boundary of a $3$-crosspolytope.
Then $\Delta$ has a missing triangle or a contractible pair $(p,q)$.
\end{theorem}

\section{Lefschetz properties of $2$-spheres}\label{sect:Lefschetz}

In this section, we study the strong Lefschetz property of simplicial $2$-spheres.
Throughout this section,
we assume that $\mathrm{char}(\FF)$ is not $2$ or $3$.

Let $\Delta$ be a simplicial $2$-sphere, and let $\Theta=\theta_1,\theta_2,\theta_3$ be an l.s.o.p.\ for $\FF[\Delta]$.
Then
$A= \FF[\Delta]/\Theta \FF[\Delta]$ is a Gorenstein algebra with
$$A=A_0 \oplus A_1 \oplus A_2 \oplus A_3$$
and $A_0 \cong A_3 \cong \FF$ (see \cite[II, Section 6]{Stan:Gr}).
Since any monomial of degree $3$ in $\FF[x_1,\dots,x_n]$ can be written as a linear combination of cubics of linear forms
if $\mathrm{char}(\FF)$ is not $2$ or $3$,
$\{ w^3: w \in A_1\}$ spans $A_3$.
Since $A_3$ is non-zero,
this implies that $\times w^3 : A_0 \to A_3$ is bijective for a generic $w$.
Thus $A$ has the SLP if and only if there is a linear form $w$ such that
$$\times w :A_1 \to A_2$$
is bijective.
Moreover, since $A_1 \cong A_2$ as $\FF$-vector spaces,
to prove the above bijectivity,
it suffices to prove that the multiplication map $\times w:A_1 \to A_2$ is surjective.
Thus, in this setting, $A$ has the SLP if and only if
$$\big(\FF[\Delta]/(\Theta,w)\FF[\Delta]) \big)_2=0$$
for some linear form $w$.

Let $\Delta$ be a simplicial complex.
We identify linear forms in $S=\FF[x_v:v\in V(\Delta)]$ with their image in $\FF[\Delta]$.
Also, for a subcomplex $\Gamma$ of $\Delta$,
we often regard $\FF[\Gamma]$ as an $S$-module.
Since there is a surjection $\FF[\Delta]/\Theta \FF[\Delta] \to \FF[\Gamma]/\Theta \FF[\Gamma]$ for any sequence $\Theta=\theta_1,\dots,\theta_k \in S$ if $\Gamma \subseteq \Delta$, the following property holds.

\begin{lemma}
\label{3.1}
Let $\Delta$ be a simplicial complex, and let $\Gamma$ be a subcomplex of $\Delta$ having the same dimension as $\Delta$. Then every l.s.o.p.\ for $\FF[\Delta]$ is an l.s.o.p.\ for $\FF[\Gamma]$.
\qed
\end{lemma}

The next statement was proved by Babson and Nevo \cite[Theorem 6.1]{Babson-Nevo}.

\begin{lemma}[Babson--Nevo]
\label{3.2}
Let $\Delta=\Gamma_1 \#_\sigma \Gamma_2$ be a simplicial $2$-sphere, $\Theta=\theta_1,\theta_2,\theta_3$ a common l.s.o.p.\ for $\FF[\Delta]$ and $\FF[\langle \sigma \rangle]$,
and let $w$ be a linear form in $\FF[x_v:v\in V(\Delta)]$.
If $(\FF[\Gamma_i]/(\Theta,w) \FF[\Gamma_i])_2=0$ for $i=1,2$, then
$(\FF[\Delta]/(\Theta,w) \FF[\Delta])_2=0$.
\end{lemma}

Recall that a balanced simplicial $2$-sphere $\Delta$ is said to have the colored SLP over $\FF$ if $\FF[\Delta]/\Theta \FF[\Delta]$ has the SLP,
where $\Theta$ is the colored s.o.p.\ for $\FF[\Delta]$.
Lemma \ref{3.2}, applied to the case that $\Theta$ is the colored s.o.p.\ implies the following corollary.

\begin{corollary}
\label{3.3}
Let $\Delta=\Gamma_1 \#_\sigma \Gamma_2$ be a balanced simplicial $2$-spheres.
If both $\Gamma_1$ and $\Gamma_2$ have the colored SLP over $\FF$, then so does $\Delta$.
\qed
\end{corollary}

We need two more technical statements.

\begin{lemma}
\label{3.5}
Let $\Delta$ be a $2$-dimensional simplicial complex, $\{s,w\},\{t,w\} \in \Delta$,
and $u$ a vertex which is not in $\Delta$.
Let $\Sigma =\langle \{s,w,u\},\{t,w,u\} \rangle$, $\Gamma=\Delta \cup \Sigma$,
$\Theta$ an l.s.o.p.\ for $\FF[\Gamma]$, and let $w$ be a linear form in $\FF[x_v:v \in V(\Gamma)]$.
If $(\FF[\Delta]/(\Theta,w)\FF[\Delta])_2=0$
and $w$ is non-zero in $\FF[\Sigma]/\Theta \FF[\Sigma]$, then
$(\FF[\Gamma]/(\Theta,w)\FF[\Gamma])_2=0$.
\end{lemma}

\begin{proof}
Let $S=\FF[x_v:v \in V(\Gamma)]$.
We have the following exact sequence of $S$-modules
$$0 \longrightarrow \FF[\Sigma] \stackrel{\times x_u\ } \longrightarrow \FF[\Gamma] \longrightarrow \FF[\Delta] \longrightarrow 0.$$
By the right-exactness of the tensor product,
tensoring the above exact sequence with $S / (\Theta,w) S$ yields the exact sequence
\begin{align}
\label{exact1}
(\FF[\Sigma]/(\Theta,w) \FF[\Sigma])_1
\stackrel {\times x_u} \to & (\FF[\Gamma]/(\Theta,w) \FF[\Gamma])_2 \to (\FF[\Delta]/(\Theta,w) \FF[\Delta])_2 \to 0.
\end{align}
From
$$\FF[\Sigma]/\Theta\FF[\Sigma] =\FF[x_s,x_t,x_u,x_w]/(x_sx_t,\theta_1,\theta_2,\theta_3) \cong \FF[x]/(x^2),$$
we infer $(\FF[\Sigma]/(\Theta,w)\FF[\Sigma])_1=0$,
if $w$ is non-zero in $\FF[\Sigma]/\Theta\FF[\Sigma]$.
Now, the desired property follows from \eqref{exact1} and the assumption
$(\FF[\Delta]/(\Theta,w)\FF[\Delta])_2=0$.\end{proof}

The following statement is crucial in our proof of Theorem \ref{thm:main1}.

\begin{lemma}
\label{3.7}
Let $\Delta$ be a balanced simplicial $2$-sphere, and let
$(p,q)$ be a contractible pair in $\Delta$.
If $\bcont(\Delta)$ has the colored SLP, then
$\Delta$ has the colored SLP.
\end{lemma}

\begin{proof}
Let $V=V(\Delta)$ be the vertex set of $\Delta$,
$S=\FF[x_v: v \in V]$, and let $\kappa$ be a proper $3$-coloring of $\Delta$. Let $\Theta=\theta_1,\theta_2,\theta_3$ be the colored s.o.p.\ for $\FF[\Delta]$, i.e.,
$\theta_i=\sum_{v \in V,\ \kappa(v)=i} x_v$ for $i=1,2,3$.
Let 
$s,t,w$ be the vertices with $\st_\Delta(p) \cap \st_\Delta(q)= \langle \{s,w\},\{w,t\}\rangle$,
and let
\begin{align*}
\Gamma =
\bcont(\Delta) \cup \langle \{ s,w,q\},\{t,w,q\} \rangle \cup \langle \{ s,w,p\},\{t,w,p\} \rangle.
\end{align*}
As $\kappa$ is also a proper coloring for $\Gamma$, $\Theta$ is also the colored s.o.p.\ for $\FF[\Gamma]$.
Since $\{s,q\},\{t,q\} \in \bcont(\Delta)$ and $w,p\notin \bcont(\Delta)$,
and since $\bcont(\Delta)$ has the colored SLP by the assumption, applying Lemma \ref{3.5} twice,
yields that there is a linear form $w$ such that $(\FF[\Gamma]/(\Theta,w)\FF[\Gamma])_2=0$. In other words,
\begin{align}
\label{3-3-1}
(I_\Gamma +(\Theta,w))_2= S_2,
\end{align}
where $(\Theta,w)$ is the ideal of $S$ generated by $\Theta$ and $w$.

Let $\mathcal G=\{x_u x_v: \{u,v\} \not \in \Delta\}$ and $\overline {\mathcal G}=\mathcal G \cup\{x_v^2:v \in V\}$.
Thus $\mathcal G$ is the set of degree $2$ generators of the Stanley--Reisner ideal $I_\Delta \subseteq S$.
Note that $x_px_q \in \mathcal G$.
For $m \in \overline{\mathcal G}$ and $t \in \FF$, we define
\begin{align}
\label{3-2}
\Phi_t(m)=
\begin{cases}
m \frac {x_p} {x_q} + t m, & \mbox{ if $x_q$ divides $m$ and $m \frac {x_p} {x_q} \not \in \overline{\mathcal G}$},\\
m, & \mbox{ otherwise},
\end{cases}
\end{align}
and define the ideal
$$J(t)=( \Phi_t(m): m \in \mathcal G) \subseteq S.$$
Also, for $t \in \FF\setminus \{0\}$,
let $\varphi_t$ be the change of coordinates of $S$ defined by $\varphi_t(x_v)=x_v$ for all $v \ne q$ and $\varphi_t(x_q)= x_p+t x_q$.

We show the following {\bf claims}:

\begin{enumerate}
\item[(a)] $I_\Delta +(x_v^2:v \in V)+(\Theta)=I_\Delta+(\Theta)$.
\item[(b)] $J(0)_2=(I_\Gamma)_2$.
\item[(c)] For $t \ne 0$, $\varphi_t(I_\Delta+(x_v^2:v\in V))_2= (J(t)+(x_v^2:v \in V))_2$.
\item[(d)]
For $t \not \in \{0,1\}$,
if
$(\varphi_t(I_\Delta+(x_v^2:v\in V))+(\Theta,w))_2=S_2$
for some linear form $w$, then there is a linear form $w'$ such that $(I_\Delta+(\Theta,w'))_2=S_2$.
\end{enumerate}

\begin{proof}[Proof of the claims]
Property (a) follows from \cite[III, Proposition 4.3]{Stan:Gr}.
Since the graph of $\Gamma$ is obtained from the graph of $\Delta$ by replacing an edge $\{p,v\} \in \Delta$ with $\{q,v\}$ whenever $\{q,v\} \not \in \Delta$
(see Figure 1),
the property (b) is straightforward by the definition of $\Phi_t$.

We prove (c).
Let $\mathcal H= \{ m \in \mathcal G: m \frac {x_p} {x_q} \in \overline {\mathcal G}\}$
and $\overline {\mathcal H} = \mathcal H \cup \{x_q^2\}$.
Note that $\Phi_t(m)=\varphi_t(m)$ for any $m \in \overline {\mathcal G} \setminus \overline {\mathcal H}$.
The $\FF$-vector space
$(J(t)+(x_v^2:v \in V))_2$ is spanned by
\begin{align}
\label{3-3}
\{ \Phi_t(m): m \in \overline {\mathcal G} \setminus \overline {\mathcal H}\} \cup \overline {\mathcal H}
\end{align}
and $\varphi_t(I_\Delta+(x_v^2:v\in V))_2$ is spanned by
\begin{align}
\label{3-4}
\textstyle
\varphi_t( \overline {\mathcal G} )=
\{ \varphi_t(m): m \in \overline {\mathcal G} \setminus \overline {\mathcal H}\} \cup
\{ m \frac {x_p} {x_q} + tm : m \in \mathcal H\} \cup \{x_p^2+2tx_px_q+t^2x_q^2\}.
\end{align}
Then, since $\varphi_t(\overline {\mathcal G})$ contains $m \frac {x_p} {x_q} $ for any $m \in \mathcal H$,
$\varphi_t(I_\Delta+(x_v^2:v \in V))$ contains $\mathcal H$.
Also, since $x_p^2,x_px_q,x_q^2 \in \overline{\mathcal G}$,
$\varphi_t( \overline {\mathcal G})$ contains $x_p^2,x_p^2 + tx_px_q$, and $x_p^2+2tx_px_q+t^2x_q^2$.
Thus $\varphi_t(I_\Delta+(x_v^2:v \in V))$ contains $x_p^2,x_px_q,x_q^2$.
Then \eqref{3-3} and \eqref{3-4} show the desired equation.

Finally, we prove (d).
We may assume that $\kappa(p)=\kappa(q)=1$.
Since $\varphi_t^{-1}(x_q)=\frac 1 t (x_q-x_p)$ and $\varphi_t^{-1}(x_v)=x_v$ for $v \ne q$,
by the assumption of (d),
\begin{align}
\label{3-5}
S_2= \varphi_t^{-1}(S_2)&=
\varphi_t^{-1} \big( \varphi_t(I_\Delta+(x_v^2:v \in V))+(\Theta,w)\big)_2\\
\nonumber
&=\big(I_\Delta+(x_v^2:v\in V)+(\varphi_t^{-1}(\theta_1),\theta_2,\theta_3,\varphi_t^{-1}(w))\big)_2.
\end{align}
Since
$\varphi_t^{-1}(\theta_1)= \frac 1 t x_q +(1- \frac 1 t) x_p + \sum_{\kappa(v) =1,\ v \ne p,q} x_v$
and since $I_\Delta+(x_v^2:v \in V)$ is a monomial ideal,
by applying the change of coordinates $\psi$ of $S$ which only changes $x_p$ to $(1-\frac 1 t)^{-1} x_p$ and $x_q$ to $t x_q$,
we infer from \eqref{3-5} that
$$\big(I_\Delta +(x_v^2: v \in V) + (\theta_1,\theta_2,\theta_3,\psi\circ\varphi_t^{-1}(w)) \big)_2=\psi(S_2)=S_2.$$
Then the desired equality follows from (a).
\end{proof}

We now go back to the proof of Lemma \ref{3.7}.
For any linear form $w$, we have
\begin{align}
\label{3-6}
\dim_\FF(J(0)+(\Theta,w))_2 \leq \dim_\FF(J(t)+(\Theta,w))_2
\end{align}
for a generic choice of $t \in \FF$.
Indeed, since $(J(t)+(\Theta,w))_2$ is spanned by
$$X=\{ \Phi_t(m) : m \in \mathcal G\} \cup \{x_v \theta_i:v\in V,\ i \in\{1,2,3\}\} \cup \{x_v w : v \in V\},$$
$\dim_\FF(J(t)+(\Theta,w))_2$ is equal to the rank of the $|X| \times (\dim_\FF S_2)$-matrix $M_t$,
whose entries are the coefficients of the degree $2$ monomials of the elements of $X$.
Since we may regard the entries of $M_t$ as polynomials in $t$, we have $\rank M_t \geq \rank M_0$ for a generic choice of $t \in \FF$.
(A generic choice of $t$ makes sense as the field $\FF$ is infinite
.)

Now, by \eqref{3-3-1}, there is a linear form $w$ such that
$$(J(0)+(\Theta,w))_2=(I_\Gamma+(\Theta,w))_2=S_2,$$
where we use claim (b) for the first equality.
Thus by \eqref{3-6}
$$\big(J(t)+(x_v^2:v \in V)+(\Theta,w)\big)_2 =S_2$$
for a generic $t \in \FF$.
Then by claim (c) we have
$$\big(\varphi_t(I_\Delta+(x_v^2:v \in V))+(\Theta,w) \big)_2=S_2,$$
and by claim (d) it follows that there is a linear form $w'$ such that
$$\big(I_\Delta+(\Theta,w')\big)_2= S_2.$$
This proves $(\FF[\Delta]/(\Theta,w') \FF[\Delta])_2=0$, as desired.
\end{proof}

We now prove Theorem \ref{thm:main1}.

\begin{proof}[Proof of Theorem \ref{thm:main1}]
We prove the statement by induction on the number of vertices.
Let $\Delta$ be a balanced simplicial $2$-sphere.
Then $\Delta$ has at least $6$ vertices, since there are at least two vertices in each color.
If $\Delta$ has exactly $6$ vertices, then $\Delta$ must be the boundary of a $3$-crosspolytope and hence
$$\FF[\Delta]/\Theta \FF[\Delta]\cong \FF[x,y,z]/(x^2,y^2,z^2),$$
which has the SLP if $\mathrm{char}(\FF)$ is not $2$ or $3$,
where $\Theta$ is the colored s.o.p.

Suppose that $\Delta$ has at least $7$ vertices.
By Theorem \ref{2.7}, either $\Delta=\Gamma \#_\sigma \Sigma$ for some balanced simplicial $2$-spheres $\Gamma$ and $\Sigma$, or there is a contractible pair $(p,q)$ in $\Delta$.
In the former case, since $\Gamma$ and $\Sigma$ have the colored SLP by the induction hypothesis,
$\Delta$ also has the colored SLP by Corollary \ref{3.3}.
In the latter case,
$\bcont(\Delta)$ has the colored SLP by the induction hypothesis,
and Lemma \ref{3.7} shows that $\Delta$ has the colored SLP.
\end{proof}

\section{$(2,1)$-balanced simplicial spheres}\label{sect:aBalanced}

In this section, we prove Theorem \ref{thm:main2}.
To simplify the argument,
we slightly modify some notation from the introduction.

Let $\Delta$ be a $2$-dimensional simplicial complex.
A {\em bi-coloring} of $\Delta$ is a map $\pi : V(\Delta) \to \{b,r\}$, where $b$ and $r$ are letters.
For a fixed bi-coloring $\pi$,
vertices $v$ with $\pi(v)=b$ (resp.\ $\pi(v)=r$) are called {\em blue vertices} (resp.\ {\em red vertices}).
A bi-coloring $\pi$ is said to be a {\em $(2,1)$-coloring} of $\Delta$ if every face $\sigma \in \Delta$ has at most two blue vertices and at most one red vertex.
Thus a $2$-dimensional simplicial complex is $(2,1)$-balanced if it has a $(2,1)$-coloring.

Given a fixed bi-coloring $\pi$ of $\Delta$,
a linear form $\theta=\sum_{v \in V(\Delta)} \alpha_v x_v \in \FF[\Delta]$ is said to be {\em blue} (resp.\ {\em red}) if $\alpha_v =0$ for all $v$ with $\pi (v) \ne b$ (resp.\ $\pi(v) \ne r$).
A {\em $(2,1)$-colored sequence} in $\FF[\Delta]$ (w.r.t.\ $\pi$) is a sequence of linear forms $\theta_1,\theta_2,\theta_3$ in $\FF[\Delta]$ such that $\theta_1,\theta_2$ are blue and $\theta_3$ is red.
If $\pi$ is a $(2,1)$-coloring of $\Delta$, then, by a result of Stanley \cite[Theorem 4.1]{Stan:79}, there is a $(2,1)$-colored sequence which is an l.s.o.p.\ for $\FF[\Delta]$.
We call such an l.s.o.p.\ a {\em $(2,1)$-colored s.o.p.} for $\FF[\Delta]$.

Recall from the previous section that, for a simplicial $2$-sphere $\Delta$ and an l.s.o.p.\ $\Theta$ for $\FF[\Delta]$, the algebra $\FF[\Delta]/\Theta \FF[\Delta]$ has the SLP if there is a linear form
$w$ such that
$$(\FF[\Delta]/(\Theta,w)\FF[\Delta])_2=0.$$
We denote by $e(\Delta)$ the number of edges of $\Delta$.
The next statement proves the implication ``(i)$\Rightarrow$(ii)'' of Theorem \ref{thm:main2}.

\begin{lemma}
\label{4.1}
Let $\Delta$ be a simplicial complex, $\pi$ a bi-coloring of $\Delta$, and $\Theta$ a $(2,1)$-colored sequence in $\FF[\Delta]$.
For any set $W$ of blue vertices of $\Delta$ with $|W| \geq 2$ and for any linear form $w$, we have
$$\dim_\FF(\FF[\Delta]/(\Theta,w)\FF[\Delta])_2 \geq e(\Delta_W) -2|W|+3.$$
\end{lemma}

\begin{proof}
The surjection $\FF[\Delta] \to \FF[\Delta_W]$ induces a surjection
$$\FF[\Delta]/(\Theta,w) \FF[\Delta] \to \FF[\Delta_W]/(\Theta,w)\FF[\Delta_W].$$
Since $\Delta_W$ has no red vertices, $\theta_3$ is zero in $\FF[\Delta_W]$
and
$$\FF[\Delta_W]/(\Theta,w) \FF[\Delta_W]=\FF[\Delta_W]/(\theta_1,\theta_2,w)\FF[\Delta_W].$$
Then,
since $\dim_\FF \FF[\Delta_W]_2=e(\Delta_W)+|W|$ and $\dim_\FF \FF[\Delta_W]_1=|W|$,
it follows that
\begin{align*}
\dim_\FF (\FF[\Delta]/(\Theta,w) \FF[\Delta])_2
&=
\dim_\FF (\FF[\Delta_W]/(\theta_1,\theta_2,w) \FF[\Delta_W])_2\\
& \geq e(\Delta_W) +|W| -(3|W|-3)\\
& = e(\Delta_W)-2|W|+3,
\end{align*}
as desired.
(The ``$-3$'' term above comes from the fact that each of $\theta_1\theta_2, \theta_1w, \theta_2w \in \FF[\Delta_W]_2$ is in at least two of the ideals $x\FF[\Delta_W]$, where $x\in\{\theta_1,\theta_2,w\}$.)
\end{proof}

\begin{example}
\label{4.2}
From Lemma \ref{4.1},
we can produce $(2,1)$-balanced $2$-spheres such that $\FF[\Delta]/\Theta \FF[\Delta]$ fails to have the SLP for any $(2,1)$-colored s.o.p.\ $\Theta$ for $\FF[\Delta]$.

Let $\Gamma$ be a simplicial $2$-sphere with $n$ vertices, and let $\Delta$ be the simplicial $2$-sphere obtained from $\Gamma$ by subdividing all facets of $\Gamma$.
Then $\Delta$ is $(2,1)$-balanced and has a unique $(2,1)$-coloring $\pi$, which is defined by $\pi(v)=b$ if $v$ is a vertex of $\Gamma$ and $\pi(v)=r$ otherwise.
Figure \ref{Fig:CounterExample} below shows the graph of $\Delta$ if $\Gamma$ is the boundary of a simplex.

\begin{center}
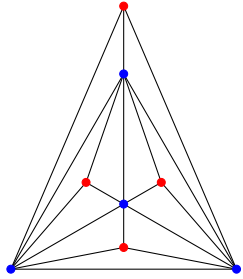
\begin{figure}[h]\label{Fig:CounterExample}
\begin{tikzpicture}
\tikzstyle{point}=[ fill=blue,draw=blue,circle, inner sep=0pt,minimum width=3pt,minimum height=3pt]
\draw (0,0) to (3,0) to (1.5,2.598) to (0,0) to (1.5,0.866) to (1.5,2.598);
\draw (1.5,0.866) to (3,0);
\draw (0,0) to (1.5,0.2887) to (3,0);
\draw (1.5,0.2887) to (1.5,0.866);
\draw (0,0) to (1,1.1547) to (1.5,2.598);
\draw (1,1.1547) to (1.5,0.866);
\draw (3,0) to (2,1.1547) to (1.5,0.866);
\draw (2,1.1547) to (1.5,2.598);
\draw (0,0) to (1.5,3.5) to (3,0);
\draw (1.5,3.5) to (1.5,2.598);
\node [point, label={[label distance=0cm]180:}] at (0,0){};
\node [point, label={[label distance=0cm]0:}] at (3,0){};
\node [point, label={[label distance=0cm]90:}] at (1.5,2.598){};
\node [point, label={[label distance=0cm]90:}] at (1.5,0.866){};
\tikzstyle{point}=[ fill=red,draw=red,circle, inner sep=0pt,minimum width=3pt,minimum height=3pt]
\node [point, label={[label distance=0cm]90:}] at (1.5,0.2887){};
\node [point, label={[label distance=0cm]90:}] at (1,1.1547){};
\node [point, label={[label distance=0cm]90:}] at (2,1.1547){};
\node [point, label={[label distance=0cm]90:}] at (1.5,3.5){};
\end{tikzpicture}
\caption{The graph of the $(2,1)$-balanced sphere constructed from a tetrahedron.}
\end{figure}
\end{center}

Let $W$ be the set of all blue vertices of $\Delta$.
Then $\Delta_W$ is the graph of $\Gamma$, so $|W|=n$ and $e(\Delta_W)=3n-6$.
Hence Lemma \ref{4.1} says that
$$\dim_\FF (\FF[\Delta]/(\Theta,w) \FF[\Delta])_2 \geq 3n-6-(2n-3) =n-3$$
for any $(2,1)$-colored s.o.p.\ $\Theta$ for $\FF[\Delta]$ and any linear form $w$.
As $n>3$,
$\FF[\Delta]/\Theta \FF[\Delta]$ fails to have the SLP for any $(2,1)$-colored s.o.p.\ $\Theta$ for $\FF[\Delta]$.
\end{example}
For any simplicial $2$-sphere $\Delta$,
by the Hard Lefschetz theorem, $\FF[\Delta]/\Theta \FF[\Delta]$ has the SLP for a generic l.s.o.p.\ $\Theta$ for $\FF[\Delta]$.
However, the previous example shows that,
for a specific choice of a simplicial $2$-sphere $\Delta$ and a specific l.s.o.p.\ $\Theta$, the dimension
$\dim_\FF (\FF[\Delta]/(\Theta,w) \FF[\Delta])_2$,
where $w$ is a generic linear form,
can be arbitrarily big.

In the rest of this section, we prove the implication ``(ii)$\Rightarrow$(i)'' of Theorem \ref{thm:main2}.
We actually consider a more general class of simplicial spheres that properly contains $(2,1)$-balanced simplicial $2$-spheres.
We say that a bi-coloring $\pi$ of a simplicial complex $\Delta$ is {\em semi-proper} if there are no edges $\{u,v\} \in \Delta$ with $\pi(u)=\pi(v)=r$.
Note that any $(2,1)$-coloring is semi-proper, but the converse is false since a semi-proper bi-coloring does not forbid the existence of a $2$-face all of whose vertices are blue.
By the Kind--Kleinschmidt's criterion on linear systems of parameters for Stanley--Reisner rings \cite[III, Lemma 2.4]{Stan:Gr},
we get the following lemma.

\begin{lemma}
\label{4.3}
Let $\Delta$ be a $2$-dimensional simplical complex, and let $\pi$ be a semi-proper bi-coloring of $\Delta$.
Then, for a generic choice of blue linear forms $\theta_1,\theta_2$ and for a generic linear form $\theta_3$,
the sequence $\theta_1,\theta_2,\theta_3$ is a system of parameters for $\FF[\Delta]$.
\end{lemma}

\begin{proof}
Let $\Theta=\theta_1,\theta_2,\theta_3 \in \FF[\Delta]$ be a sequence of linear forms with $\theta_i=\sum_{v \in V(\Delta)} \alpha_{i,v} x_v$.
The Kind--Kleinschmidt's criterion says that if, for any face $\sigma \in \Delta$,
the matrix $(\alpha_{i,v})_{1 \leq i \leq 3,\ v \in \sigma}$ has rank $|\sigma|$, then $\Theta$ is an l.s.o.p.\ for $\FF[\Delta]$.
Since each face has at most one red vertex, if we choose $\Theta$ generically under the restriction that $\alpha_{i,v}=0$ when $\pi(v)=r$ and $i \in \{1,2\}$,
then the Kind--Kleinschmidt's criterion shows that $\theta_1,\theta_2,\theta_3$ is an l.s.o.p.\ for $\FF[\Delta]$.
\end{proof}


Next, we prove analogues of Corollary \ref{3.3} and Lemma \ref{3.7} for the semi-proper setup.
Let $\Delta$ be a simplicial $2$-sphere with a semi-proper bi-coloring $\pi$.
We say that $\Delta$ has the {\em $\pi$-colored SLP} (over $\FF$) if there are a $(2,1)$-colored sequence $\Theta=\theta_1,\theta_2,\theta_3$ in $\FF[\Delta]$ and a linear form $w$ such that
$$(\FF[\Delta]/(\Theta,w) \FF[\Delta])_2 =0.$$
Note that if $\Delta$ has the $\pi$-colored SLP, then $(\FF[\Delta]/(\Theta,w) \FF[\Delta])_2=0$ for a generic choice of a $(2,1)$-colored sequence $\Theta$ and a generic linear form $w$.
In particular, if $\pi$ is a $(2,1)$-coloring, then $\Theta$ can be taken as an l.s.o.p.\ for $\FF[\Delta]$.

\begin{lemma}
\label{4.4}
Let $\Delta= \Gamma_1 \#_\sigma \Gamma_2$ be a simplicial $2$-sphere with a semi-proper bi-coloring $\pi$.
If both $\Gamma_1$ and $\Gamma_2$ have the $\pi$-colored SLP, then so does $\Delta$.
\end{lemma}

\begin{proof}
Let $S=\FF[x_v: v \in V(\Delta)]$.
If we choose a $(2,1)$-colored sequence $\theta_1,\theta_2,\theta_3 \in S$ and a linear form $w \in S$ generically,
then $\theta_1,\theta_2,w$ is a common system of parameters for $\FF[\Delta]$ and $\FF[\langle \sigma \rangle]$ by Lemma \ref{4.3}.
Also, $(\FF[\Gamma_i]/(\Theta,w)\FF[\Gamma_i])_2=0$ for $i \in \{1,2\}$ by the assumption.
Then the assertion follows from Lemma \ref{3.2}.
\end{proof}

\begin{lemma}
\label{4.5}
Let $\Delta$ be a simplicial $2$-sphere with a semi-proper bi-coloring $\pi$, and let $\{p,q\}\in \Delta$ with $\pi(p)=\pi(q)=b$. Assume that $\st_\Delta(p) \cap \st_\Delta(q)$ is an induced subcomplex of $\Delta$ consisting of two triangles $\langle \{s,p,q\},\{t,p,q\} \rangle$ and that $\pi(s)=r$.
Then
$\Delta \to \cont (\Delta)$ is an admissible contraction
and if $\cont(\Delta)$ has the $\pi$-colored SLP, then $\Delta$ has the $\pi$-colored SLP.
\end{lemma}

\begin{proof}
The proof is similar to that of Lemma \ref{3.7}.
That the contraction $\Delta \to \cont (\Delta)$ is admissible is obvious.
Let $S=\FF[x_v : v \in V(\Delta)]$.
Let
\begin{align}
\label{4-1}
\Sigma:= \st_\Delta(p) \cap \st_\Delta(q)= \langle \{s,p,q\},\{t,p,q\} \rangle,
\end{align}
and let
$$\Gamma := \cont(\Delta) \cup \Sigma.$$
Note that $\pi$ gives a semi-proper bi-coloring of $\Gamma$.
For a generic choice of blue linear forms $\theta_1,\theta_2$, of a red linear form $\theta_3$ and of a linear form $w \in S$,
the sequence $\theta_1,\theta_2,w$ is an l.s.o.p.\ for $\FF[\Gamma]$ by Lemma \ref{4.3}.
Moreover, $\theta_3$ is non-zero in $\FF[\Sigma]/(\theta_1,\theta_2,w)\FF[\Sigma]$,
since otherwise either $w$ is zero in
$\FF[\Sigma]/(\theta_1,\theta_2,\theta_3)\FF[\Sigma]$
or $\theta_3$ is zero in $\FF[\Sigma]/(\theta_1,\theta_2)\FF[\Sigma]$; none of
which can happen as $\Sigma$ has a red vertex.

Then, by Lemma \ref{3.5} and the assumption that $\cont(\Delta)$ has the $\pi$-colored SLP, we have
\begin{align}
\label{4-2}
\big(S/(I_\Gamma+(\Theta,w))\big)_2= \big(\FF[\Gamma]/(\Theta,w)\FF[\Gamma]\big)_2=0,
\end{align}
where $\Theta=\theta_1,\theta_2,\theta_3$.

Let $\mathcal G= \{ x_ux_v: \{u,v\} \not \in \Delta\}$.
For $m \in \mathcal G$ and $t \in \FF$,
we define $\Phi_t(m)$ in the same way as in \eqref{3-2}.
Also, for $t \in \FF \setminus \{0\}$, let $\varphi_t$ be the change of coordinates of $S$
defined by
$\varphi_t(x_v)=x_v$ for $v \ne q$ and $\varphi_t(x_q)=x_p+tx_q$.
Let $J(t)=(\Phi_t(m): m \in \mathcal G)$.
Then it is not hard to prove that
\begin{enumerate}
\item[(a)] $J(0)_2=(I_\Gamma)_2$, and
\item[(b)] $(\varphi_t(I_\Delta))_2=(J(t))_2$ for $t \ne 0$.
\end{enumerate}
Indeed, (a) easily follows from \eqref{4-1}, and (b) follows from a similar (and simpler) argument as claim (c) in the proof of Lemma \ref{3.7}.

Now, using (b), for a generic $t \in \FF$, we have
$$\dim_\FF(J(0)+(\Theta,w))_2\leq \dim_\FF (J(t) + (\Theta,w))_2 = \dim_\FF(I_\Delta+(\varphi_t^{-1}(\Theta), \varphi_t^{-1}(w)))_2.$$
Since $(J(0)+(\Theta,w))_2=(I_\Gamma+(\Theta,w))_2=S_2$ by (a) and \eqref{4-2},
the above inequality shows
$$\big(S/\big(I_\Delta+(\varphi_t^{-1}(\Theta),\varphi_t^{-1}(w)) \big)\big)_2=0.$$
Since $(\varphi_t^{-1}(\theta_1), \varphi_t^{-1}(\theta_2),\varphi_t^{-1}(\theta_3))$ is a $(2,1)$-colored sequence,
the above equation proves that $\Delta$ has the $\pi$-colored SLP.
\end{proof}

The following theorem completes the proof of the remaining part of Theorem \ref{thm:main2}.

\begin{theorem}
\label{4-6}
Let $\Delta$ be a simplicial $2$-sphere with a semi-proper bi-coloring $\pi$ that satisfies the following property (L):\medskip

(L) $e(\Delta_W) \leq 2 |W|-3$ for any set $W$ of blue vertices with $|W| \geq 2$.
\medskip

\noindent
Then $\Delta$ has the $\pi$-colored SLP.
\end{theorem}

\begin{proof}
We proceed by induction on $|V|$ for $V=V(\Delta)$.
If $|V|=4$, then $\Delta$ is the boundary of a tetrahedron and as (L) holds $\Delta$ has $3$ blue vertices and one red vertex w.r.t.\ $\pi$. One readily verifies that $\Delta$ has the $\pi$-colored SLP.

Assume $|V|>4$.
If $\Delta$ has a missing triangle, then, by Lemma \ref{missingface}, $\Delta$ decomposes as a connected sum $\Delta=\Gamma_1\#_\sigma \Gamma_2$. In this case, $\pi$ induces a semi-proper bi-coloring on each $\Gamma_i$, and clearly (L) holds for each $\Gamma_i$. Hence, by the induction hypothesis, each $\Gamma_i$ has the $\pi$-colored SLP, and thus, by Lemma~\ref{4.4}, so has $\Delta$.

Thus, assume $\Delta$ has no missing triangle.
Then, for any edge $\{p,q\} \in \Delta$, $\Delta \to \cont(\Delta)$ is an admissible contraction. Moreover, if $p$ and $q$ are blue vertices,
then $\pi$ induces a semi-proper coloring of $\cont(\Delta)$.
We will show that there is a facet $\{p,q,r\}$ with $\pi(p)=\pi(q)=b$ and $\pi(r)=r$ such that the complex $\cont(\Delta)$ satisfies (L).
Then, by the induction hypothesis, $\cont(\Delta)$ has the $\pi$-colored SLP, and thus, by Lemma~\ref{4.5} so has $\Delta$, as desired.

%
%

We distinguish two cases: whether $\Delta$ contains a blue facet (i.e., a facet all of whose vertices are blue) or not.

\textbf{Case (i)}: Assume $\Delta$ has no blue facet.
Recall that $\Delta$ has no missing triangle either.
Then, for every subset $W$ of blue vertices,
the $1$-skeleton of $\Delta_W$ has no $3$-cycles, thus
by Euler's formula
$e(\Delta_W)\leq 2|W|-4$, whenever $|W|\geq 3$.
Since $e(\cont(\Delta)_W) =e(\Delta_W)$ if $q \not \in W$ and $e(\cont(\Delta)_W) \leq e(\Delta_{W \cup \{p\}})-1$ if $q \in W$,
this implies that condition (L) holds in $\cont(\Delta)$ for \emph{any} blue edge $\{p,q\}\in \Delta$. 

\textbf{Case (ii)}: Assume $\Delta$ has a blue facet.
We first show that there is a blue facet $T=\{v_1,v_2,v_3\}$ such that
there exist red vertices $v_1',v_2'$ (possibly $v_1'=v_2'$) with $\{v_1,v_1',v_3\},\{v_2,v_2',v_3\} \in \Delta$.
Then we proceed to show that either $\mathcal{C}_{v_1 \to v_3}(\Delta)$ or $\mathcal{C}_{v_2 \to v_3}(\Delta)$ satisfies (L), for \emph{some} such choice.

Suppose to the contrary that there is no blue facet $T$ satisfying the above condition.
This means that each blue facet is adjacent to at least two blue facets in the dual graph of $\Delta$.
Consider the graph $G$ whose vertices are the blue facets of $\Delta$ and two facets $\sigma,\tau$ are adjacent if their intersection is an edge of $\Delta$.
Then each vertex of $G$ has degree at least two, and therefore $G$ has an induced cycle $\sigma_1,\dots,\sigma_k$.
Let $\Gamma=\langle \sigma_1,\dots,\sigma_k\rangle$ and $W=V(\Gamma)$.
Then, since we take an induced cycle in $G$,
there are exactly $k$ edges, which are contained in two facets in $\Gamma$, which implies $e(\Gamma)=3k-k=2k$.
Also, since $|V(\Gamma)|=|V(\langle \sigma_1,\dots,\sigma_{k-1}\rangle)|$ and $|V(\langle \sigma_1,\dots,\sigma_{i}\rangle)|-|V(\langle \sigma_1,\dots,\sigma_{i-1}\rangle)|\leq 1$ for $i<k$,
we have $|V(\Gamma)| \leq k+1$.
Thus we have $e(\Delta_W) \geq e(\Gamma) =2k \geq 2|W|-2$ which contradicts (L).

Let $T,v_1,v_2,v_3, v'_1,v'_2$ be as guaranteed above.
Next we show that either at least one of the complexes $C_1:=\mathcal{C}_{v_1 \to v_3} (\Delta)$ and $C_2:=\mathcal{C}_{v_2 \to v_3}(\Delta)$ satisfies (L), or we are in a situation that allows an inductive argument to find some other choice as above for which one of $C_1$ and $C_2$ does satisfy (L).
Assume both $C_1$ and $C_2$ violate (L).
Then for $i=1,2$, there is a subset of blue vertices $B'_i$ in $C_i$ with $e((C_i)_{B'_i})>2|B'_i|-3$. In particular, the vertex $v_3$ is in $B'_i$, and for the set $B_i=B'_i\cup\{v_i\}\subseteq V$
we must have
(i) $e(\Delta_{B_i})=2|B_i|-3$ and (ii) $v_{3-i}$ is not in $B_i$;
this is because $\Delta_{B_i}$ satisfies the inequality in (L) and $(C_i)_{B'_i}$ violates it.

Consider the union $B=B_1 \cup B_2$.
Now we count edges in $\Delta_{B_1}\cup \Delta_{B_2}$:
if $|B_1\cap B_2|\geq 2$, then
\begin{align*}
e(\Delta_{B_1}\cup \Delta_{B_2}) =&e(\Delta_{B_1})+e(\Delta_{B_2})-e(\Delta_{B_1\cap B_2})\\
=& 2|B_1|-3 + 2|B_2|-3 - e(\Delta_{B_1\cap B_2})\\
\geq &2|B_1|-3 + 2|B_2|-3 - (2|B_1\cap B_2|-3)
= 2|B|-3.
\end{align*}
The edge $\{v_1,v_2\}$ is in the $1$-skeleton of $\Delta_B$ but not in $\Delta_{B_1}\cup \Delta_{B_2}$.
Thus $\Delta_B$ violates the inequality in (L), a contradiction.
This completes the proof, unless $|B_1\cap B_2|\leq 1$, in which case $B_1\cap B_2=\{v_3\}$.

Call a subset $U$ of $V$ \emph{Laman} if the complex $\Delta_{U}$ satisfies (L) and $e(\Delta_U)=2|U|-3$. Then $B_1$ is Laman, $v_1,v_3\in B_1$ and $v_2\notin B_1$ (so $C_1$ violates (L)); let $B_1$ be of maximal size with these properties.

Next, we show that there is a blue facet $T"=\{u_1,u_2,u_3\}\subseteq B_1$ such that each of the edges $u_1u_3$ and $u_2u_3$ is contained in a facet whose third vertex is red; note that none of these two edges is $v_1v_3$.
Indeed,
in order to apply the argument used in Case (ii), to $\Delta_{B_1}$ rather than to $\Delta$, what we need to verify is that if $F\subseteq B_1$ is a blue facet adjacent in $\Delta$ to another facet $F'$, and $\{z\}=F'\setminus F$, then $z\in B_1$. Now, if $z\notin B_1$, then $B_1\cup \{z\}$ is Laman,
so by maximality of $B_1$ we must have $z=v_2$, but one of the edges $v_2v_3, v_2v_1$ is not in $\Delta_{B_1}\cup F'$, thus $\Delta_{B_1\cup \{z\}}$ violates (L), a contradiction.

As argued before, if both $\mathcal{C}_{u_1 \to u_3} (\Delta)$ and $\mathcal{C}_{u_2 \to u_3} (\Delta)$ violate (L), then there exist for $i=1,2$ Laman subsets $B_i''$ with $u_i,u_3\in B_i''$ and $u_{3-i}\notin B_i''$. If $|B_1''\cap B_2''|\geq 2$, then the claim follows by the same computation as above. So, assume $B_1''\cap B_2''=\{u_3\}$.
Next we show that in this case $B_i''\subset B_1$, for at least one of $i=1,2$; the inclusion is strict.

Note that $|B_i''\cap B_1|\geq 2$. Thus, a count of edges, similar to the above in Case (ii), gives that $B_i''\cup B_1$ is Laman with $\Delta_{B_i''\cup B_1}$ and $\Delta_{B_i''}\cup \Delta_{B_1}$ having the same $1$-skeleton. By the maximality of $B_1$, for each of $i=1,2$, either $B_i''\subset B_1$ (with strict containment as $u_{3-i}\in B_1\setminus B_i''$) or $v_2\in B_i''$. The latter case cannot happen for both $i=1,2$ as $u_3\neq v_2$; thus we can assume $B_1''\subset B_1$, and we choose such $B_1''$ of maximal size.

As $|B_1''|<|B_1|$, by iterating this argument for $B_1''$ and inductively, we conclude that at some point an edge $\{x,y\}$ is found that is contained in a unique blue facet and such that $\mathcal{C}_{x \to y}(\Delta)$ satisfies (L).
This completes the proof.
\end{proof}


\bigskip
\bigskip
\noindent
\textbf{Acknowledgments}:
We thank Yusuke Suzuki for letting us know Batagelj's result.
A substantial part of
this work was done during
the workshop ``Lefschetz Properties and Artinian Algebras'' at Banff International Research Station (BIRS) from March 13 to 18, 2016.
We thank the organizers of the workshop and BIRS for their kind invitation and warm hospitality.


\begin{thebibliography}{1}


\bibitem[BN]{Babson-Nevo}
E. Babson and E. Nevo,
Lefschetz properties and basic constructions on simplicial spheres,
J. Algebraic Combin., vol. {\bf 31} (1), 111--129, 2010.


\bibitem[BD]{BD}
B. Bagchi and B. Datta,
Minimal triangulations of sphere bundles over the circle,
J. Combin. Theory Ser. A, vol. {\bf 115} (5), 737--752, 2008.


\bibitem[Bar]{Bar-RP2}
D. Barnette,
Generating the triangulations of the projective plane,
J. Combin. Theory Ser. B, vol. {\bf 33} (3), 222--230, 1982.

\bibitem[Bat]{Ba}
V. Batagelj,
Inductive definition of two restricted classes of triangulations,
Discrete Math., vol {\bf 52}, 113--121, 1984.

\bibitem[GY]{Golovina-Yaglom}
L. I. Golovina and I. M. Yaglom,
Induction in Geometry, Topics in Mathematics, D.C. Heath and Company, Boston, 1963. (English translation from Russian.)



\bibitem[HMMNWW]{book:Lefschetz}
T. Harima, T. Maeno, H. Morita, Y. Numata, A. Wachi and J. Watanabe,
{\em The Lefschetz properties}, Lecture Notes in Mathematics 2080, Springer, 2013.

\bibitem[La]{Laman}
G. Laman,
On graphs and rigidity of plane skeletal structures,
J. Engrg. Math., vol. {\bf 4}, 331--340, 1970.


\bibitem[Nev]{Nevo-VK}
E. Nevo,
Higher minors and Van Kampen's obstruction,
Math. Scand., vol. {\bf 101}, 161--176, 2007.


\bibitem[St1]{Stan:79}
R.P. Stanley,
Balanced Cohen-Macaulay complexes,
\textit{Trans.\ Amer.\ Math.\ Soc.} \textbf{249} (1979), 139--157.

\bibitem[St2]{Stan:Gr}
R.P. Stanley,
{\em Combinatorics and Commutative Algebra, Second Edition},
Birkh\"auser, 1996.

\end{thebibliography}
\end{document}